\documentclass[11pt,reqno,british,twoside]{article}

\usepackage[T1]{fontenc}
\usepackage[sc,noBBpl]{mathpazo}
\usepackage{babel,mathtools,amsmath,amssymb,url,paralist,amscd,microtype}
\usepackage{fixltx2e,fancyhdr,textcmds}
\usepackage[amsmath,thmmarks]{ntheorem}
\rmfamily
\mathtoolsset{mathic=true}
\setdefaultenum{\upshape(i)}{}{}{}

\theoremstyle{plain}
\theoremseparator{.}
\newtheorem{theorem}{Theorem}[section]

\newtheorem{lemma}[theorem]{Lemma}
\newtheorem{corollary}[theorem]{Corollary}

\theorembodyfont{\normalfont}
\newtheorem{definition}[theorem]{Definition}

\theoremheaderfont{\itshape}

\theoremsymbol{\begin{small}\ensuremath{\quad\triangle}\end{small}}
\newtheorem{remark}[theorem]{Remark}
\theoremsymbol{\begin{small}\ensuremath{\quad\diamondsuit}\end{small}}
\newtheorem{example}[theorem]{Example}

\theoremstyle{nonumberplain}
\theoremheaderfont{\itshape}
\theorembodyfont{\normalfont}
\theoremsymbol{\begin{small}\ensuremath{\quad\Box}\end{small}}
\newtheorem{proof}{Proof}

\qedsymbol{\begin{small}\ensuremath{\quad\Box}\end{small}}

\let\oldbibliography\thebibliography
\renewcommand{\thebibliography}[1]{%
  \oldbibliography{#1}%
  \small
  \setlength{\itemsep}{0pt}%
  \setlength{\parskip}{0pt}%
}

\numberwithin{equation}{section}

\newcommand{\itref}[1]{\eqref{#1}}

\newcommand{\Lie}[1]{\operatorname{\textsl{#1}}}
\newcommand{\lie}[1]{\operatorname{\mathfrak{#1}}}
\newcommand{\SU}{\Lie{SU}}
\newcommand{\su}{\lie{su}}

\newcommand{\GL}{\Lie{GL}}
\newcommand{\Sl}{\lie{sl}}

\newcommand{\gl}{\lie{gl}}
\newcommand{\Un}{\Lie{U}}
\newcommand{\un}{\lie{u}}
\newcommand{\g}{\lie{g}}

\newcommand{\Le}{{\mathrm{L}}}
\newcommand{\Ri}{{\mathrm{R}}}

\newcommand{\bC}{{\mathbb C}}
\newcommand{\bH}{{\mathbb H}}
\newcommand{\bR}{{\mathbb R}}

\newcommand{\CP}{\bC P}

\DeclarePairedDelimiter{\abs}{\lvert}{\rvert}
\DeclarePairedDelimiter{\norm}{\lVert}{\rVert}
\makeatletter

\DeclareMathOperator{\Id}{Id}
\DeclareMathOperator{\Hom}{Hom}

\DeclareMathOperator{\im}{Im}
\DeclareMathOperator{\diag}{diag}
\DeclareMathOperator{\rank}{rank}

\newcommand{\hkmod}[1]{{#1}_{\textup{mod}}}
\newcommand{\Mmod}{\hkmod{M}}

\newcommand{\br}{\hspace{0pt}}
\newcommand{\bdash}{-\br} %dash allowing hyphen in next word

\begin{document}
\thispagestyle{empty}

\begin{footnotesize}
  \begin{flushright}
    IMADA-PP-2010-03\\
    CP3-ORIGINS-2010-5
  \end{flushright}
\end{footnotesize}

\bigskip

\begin{center}
  \LARGE\bfseries Non-Abelian Cut Constructions and Hyperk\"ahler
  Modifications
\end{center}
\begin{center}
  \Large Andrew Dancer and Andrew Swann
\end{center}

\bigskip
\begin{abstract}
  We discuss a general framework for cutting constructions and reinterpret
  in this setting the work on non-Abelian symplectic cuts by Weitsman.  We
  then introduce two analogous non-Abelian modification constructions for
  hyperk\"ahler manifolds: one modifies the topology significantly, the
  other gives metric deformations.  We highlight ways in which the geometry
  of moment maps for non-Abelian hyperk\"ahler actions differs from the
  Abelian case and from the non-Abelian symplectic case.
\end{abstract}

\bigskip
\begin{center}
  \begin{minipage}{0.7\linewidth}
    \begin{small}
      \tableofcontents
    \end{small}
  \end{minipage}
\end{center}

\vfill
\begin{footnotesize}
  \hbox to 16em{\hrulefill}


\end{footnotesize}

\clearpage
\section{Introduction}
\label{sec:intro}

Cuts and modifications were introduced as constructions of symplectic or
hyperk\"ahler manifolds from examples in the same dimension with circle or
torus symmetry \cite{Lerman:cuts,Burns-GL:cuts,Dancer-S:mod}.  These
constructions start with a space~\( M \) that has an action of an Abelian
group~\( G \) preserving the geometric structure and admitting a moment
map.  One then chooses a space~\( X \) with the same type of geometric
structure and also with a \( G \)-action, such that the moment reduction
of~\( X \) is a point.  The reduction of \( M \times X \) by the
anti-diagonal \( G \)-action then gives a new space~\( \hat M \) of the
same dimension as~\( M \), inheriting a geometrical structure of the same
type as that on~\( M \) together with a \( G \)-action from the
\emph{diagonal} action on \( M \times X \).  For symplectic manifolds \(
\hat M \) is referred to as the \emph{cut space} of~\( M \).  For
hyperk\"ahler manifolds, where the construction has a somewhat different
character, we introduced the term \emph{modification}
in~\cite{Dancer-S:mod}.

In Lerman's original construction for symplectic manifolds with circle
action \cite{Lerman:cuts}, one takes \( X = \bC = \bR^2 \); in the case \(
G = T^n \) of~\cite{Burns-GL:cuts} the space~\( X \) is a toric variety of
real dimension~\( 2n \).  Lerman's construction removes part of~\( M \) and
collapses circle orbits on the resulting boundary to give a smooth
symplectic manifold~\( \hat M \).  In the hyperk\"ahler
setting~\cite{Dancer-S:mod}, we took \( X = \bH = \bR^4 \) when \( G \)~is
a circle, and one may take \( X \)~to be a hypertoric variety
\cite{Bielawski-Dancer:toric} of real dimension~\( 4n \) when~\( G=T^n \).
Hyperk\"ahler modifications by circles change the topology of~\( M \)
whilst preserving completeness properties of the Ricci-flat metric.  The
whole of~\( M \) plays a role, no part is removed.  In favourable
situations the modification increases the second Betti number by one and in
small dimensions the construction may be interpreted as adding a D6-brane.

It is natural to ask whether these constructions have analogues for
non-Abelian groups~\( G \).  In symplectic geometry, the case of~\( G =
\Un(n) \) and \( X = \Hom_\bC (\bC^n,\bC^n) \) was analysed by
Weitsman~\cite{Weitsman:cuts} and applied to geometric quantization.  We
provide a simple approach to his construction in terms of the polar
decomposition of matrices in~\S\ref{sec:symplectic}.

The main concern of this paper is to discuss related constructions for
hyperk\"ahler manifolds with actions of non-Abelian groups, with a
particular focus on the case \( G = \Un(n) \).  A general theme of our
approach is that the geometry of the cutting or modification is controlled
by the moment map geometry of~\( X \).  We are therefore led to consider
hyperk\"ahler moment maps for non-Abelian actions.  As we shall see, these
exhibit rather different behaviour from moment maps in the Abelian
hyperk\"ahler and the (Abelian or non-Abelian) symplectic settings.

We describe two different constructions.  In the first, \S\ref{sec:hom}, we
take \( X=\Hom_\bC(\bC^n,\bC^n) \oplus \Hom_\bC(\bC^n,\bC^n)=\bH^{n^2} \)
and only consider actions of~\( \Un(n) \).  However, there are applications
to \( \SU(n) \) manifolds.  The construction has features of both the cut
and of the modification: parts of \( M \) are removed, the modified space
contains a copy of the hyperk\"ahler reduction of the original space,
and equivariant bundle structures are
changed on open sets.  A new feature also occurs: some parts of~\( M \) are
blown-up.  One example of this construction is provided by the deformed
instanton spaces of Nakajima~\cite{Nakajima:Hilbert}.  The second
construction, \S\ref{sec:gauge}, is applicable for any compact Lie group~\(
G \) and uses the hyperk\"ahler structure found on~\( X = T^*G_\bC \)
\cite{Kronheimer:cotangent,Dancer-Swann:compact-Lie}.  The latter
construction preserves the topology but gives deformations of the
hyperk\"ahler geometry.  The paper starts by reviewing the known cut and
modification constructions.

\section{The double fibration picture}
\label{sec:double}

\paragraph{Symplectic Abelian case}

Suppose \( (M,\omega) \) is a symplectic manifold, so \( \omega \) is a
closed non-degenerate two-form.  If a group \( G \) acts on \( M \)
preserving \( \omega \), then a moment map is by definition an equivariant
map \( \mu\colon M\to\lie g^* \) such that \( d\langle\mu,A\rangle =
\xi_A{\lrcorner} \omega \) for each \( A\in\lie g \), where \( \xi_A \) is
the corresponding vector field on~\( M \), see
e.g.~\cite{Guillemin-S:symplectic}.  If \( \mu \) exists then the \( G
\)-action is said to be Hamiltonian.

Suppose \( X \) is a Hamiltonian \( G \)-manifold with moment map~\( \phi
\), that \( G \)~acts freely on an open set of~\( X \) and that the
reduction \( \phi^{-1}(\varepsilon)/G \) is (where non-empty) a point for
each regular value \( \varepsilon\in\lie z^* \), where \( \lie z \) is the
centre of \( \lie g \).  If \( G \) is Abelian, we define the cut \( \hat M
\) to be
\begin{equation*}
  \hat M = (\mu - \phi)^{-1}(\varepsilon) / G
\end{equation*}
where \( G \)~acts on \( M\times X \) via the anti-diagonal action \(
(m,x)\mapsto (g\cdot m,g^{-1}\cdot x) \) and \( (\mu-\phi)(m,x) = \mu(m)-
\phi(x) \).  Thus \( \hat M \) is the symplectic reduction of \( M\times X
\) by the anti-diagonal action of~\( G \).  If smooth, the manifold \( \hat
M \) is now a symplectic \( G \)-manifold.  In general, \( \hat M \) will
be a stratified symplectic space, cf.~\cite{Sjamaar-Lerman}.  We will work
in the smooth category.

The precise way in which \( M \) and~\( \hat M \) are related can be
understood in terms of the moment map geometry of the \( G \)-action on the
particular space~\( X \).

For~\( G=T^n \), one takes \( X \) to be a toric variety of dimension~\( 2n
\) \cite{Burns-GL:cuts}.  The moment map \( \phi \) on \( X \) has image a
convex polyhedron~\( \Delta \) in~\( \bR^n \).  Moreover \( T^n \) acts
transitively on the fibres of~\( \phi \).  In fact this map gives a trivial
\( T^n \)-fibration over the interior of~\( \Delta \).  On the boundary of
\( \Delta \), the fibres are tori of lower dimension as the \( T^n
\)-action is no longer free: in particular \( \phi \) is injective over
vertices of~\( \Delta \).

The cut manifold \( \hat M \) fits into what we call a \qq{double fibration
picture}
\begin{equation*}
  \begin{CD}
    M @<\pi_1<< N @>\pi_2>> \hat M
  \end{CD}
\end{equation*}
(though the left hand map~\( \pi_1 \) can have some special fibres), with
\( N = \{\,(m, x) : \mu(m) - \phi(x) = \varepsilon \,\} \) and \( \pi_2
\)~the quotient map for the anti-diagonal \( T^n \) action.  We assume we
have chosen~\( \varepsilon \) so that the \( T^n \)-action on~\( N \) is
free, so \( \pi_2 \) is a \( T^n \)-fibration.

The map~\( \pi_1 \) is defined by \( (m, x) \mapsto m \), so the fibre
of~\( \pi_1 \) over~\( m \) may be identified with the fibre of~\( \phi \)
over~\( \mu(m) - \varepsilon \).  Hence the image of~\( \pi_1 \) is just \(
\mu^{-1} (\Delta + \varepsilon) \).  Moreover on the interior of this set
\( \pi_1 \) is a trivial \( T^n \)-fibration, while on the boundary the
fibres are lower-dimensional tori.  So \( \hat M \) may be obtained by
removing the complement of \( \mu^{-1} (\Delta + \varepsilon) \) and
performing collapsing by the appropriate tori on the boundary.

In the particular case of Lerman's original construction we have \( G = S^1
\) and \( X = \bC \), so \( \Delta \) is the non-negative half\cwm line
in~\( \bR \).  Now \( \phi \) is injective over the origin and is a trivial
circle fibration on the positive half-line.  Then \( \hat M \) is obtained
by removing \( \mu^{-1}(- \infty, \varepsilon) \) and collapsing circle
fibres on \( \mu^{-1}(\varepsilon) \).

\paragraph{Hyperk\"ahler Abelian case}
The hyperk\"ahler case studied in \cite{Dancer-S:mod} presents some new
features.  The data of a hyperk\"ahler manifold \( M \) consists of a
metric \( g \) and three compatible complex structures \( I \), \( J \), \(
K \) with \( IJ = K = -JI \) such that the two-forms \(
\omega_I(\cdot,\cdot) = g(I\cdot,\cdot) \), \( \omega_J \) and \( \omega_K
\) are closed.  This implies that \( g \) is Ricci-flat and that \( M \) is
symplectic in three ways.  If the action of \( G \) is tri-Hamiltonian, that is, 
if \( G \) acts in a Hamiltonian way with
respect to each of the symplectic forms \( \omega_I \), \( \omega_J \), \(
\omega_K \), we consider the hyperk\"ahler moment map \( \mu\colon M\to
\lie g^*\otimes \bR^3 \) given by \( \mu = (\mu_I,\mu_J,\mu_K) \).  When \(
G \) acts freely \( \mu^{-1}(\varepsilon) / G \) is again a hyperk\"ahler
manifold of dimension \( 4\dim G \) less than \( M \) for each \(
\varepsilon\in\lie z^*\otimes \bR^3 \).

For \( G \)~Abelian, the hyperk\"ahler modification is defined in an
analogous way to the symplectic cut: one takes \( X \) to be an appropriate
hyperk\"ahler manifold with tri-Hamiltonian \( G \)-action, moment map~\(
\phi\colon X\to\lie g^*\otimes\bR^3 \), and puts
\begin{equation*}
  \Mmod = (\mu-\phi)^{-1}(\varepsilon) / G
\end{equation*}
for the anti-diagonal \( G \)-action on \( M\times X \).  Now consider the
double fibration picture
\begin{equation*}
  \begin{CD}
    M @<\pi_1<< N @>\pi_2>> \Mmod.
  \end{CD}
\end{equation*}

For circle actions, we take \( X = \bH = \bR^4 \), and now \( \phi \colon X
\rightarrow \bR^3 \) is \( \phi(q) = \overline q i q \), which is
surjective, injective over the origin and a non-trivial circle fibration
away from the origin.  (In fact, it is the Hopf map).  This gives that the
map~\( \pi_1 \) is surjective, is injective over \( \mu^{-1}(\varepsilon)
\), and is a non-trivial circle fibration away from \(
\mu^{-1}(\varepsilon) \).

Hence we still collapse circle fibres on the set~\( \mu^{-1}(\varepsilon)
\) but do not now remove any part of~\( M \).  The complements \( M
\setminus \mu^{-1}(\varepsilon) \) and \( \Mmod \setminus
(\mu^{-1}(\varepsilon)/S^1) \) are not diffeomorphic; rather they have a
common space \( N \setminus (\mu^{-1}(\varepsilon) \times \{0\} ) \)
sitting above them as the total space of non-trivial circle fibrations.

\paragraph{Non-Abelian constructions}

Let us now consider the case of non-Abelian \( G \).  As the diagonal and
anti-diagonal actions no longer commute, the above method will no longer
produce a space with a \( G \)-action.  One circumvents this problem by
instead requiring~\( X \) to have a \( G \times G \)-action.  We denote the
left and right copies of~\( G \) by \( G_\Le \) and~\( G_\Ri \)
respectively.

One can then define~\( \hat M \) (or \( \Mmod \)) to be the moment
reduction of \( M \times X \) by \( G \) at level \( \varepsilon\in\lie z^*
\) (respectively \( \lie z^*\otimes\bR^3 \)), where \( G \) acts in the
given way on~\( M \) and by \( G_\Ri \) on~\( X \).  The action of~\( G_\Le
\) now induces a \( G \)-action on~\( \hat M \).  Again, we have a double
fibration picture
\begin{equation}
  \label{eq:double}
  \begin{CD}
    M @<\pi_1<< N @>\pi_2>> \hat M, \Mmod
  \end{CD}
\end{equation}
where
\begin{equation*}
  N = \{\, (m, x) : \mu(m) + \phi (x) = \varepsilon \,\}
\end{equation*}
and \( \phi \) is the moment map for the \( G_\Ri \)-action on~\( X \).  As
in the Abelian case, to analyse this picture we must understand the fibres
of \( \phi \).

\section{Symplectic cuts by unitary groups}
\label{sec:symplectic}

Let us now study from the above point of view Weitsman's construction of
cuts for symplectic manifolds with \( \Un(n) \) action
\cite{Weitsman:cuts}.  We shall find that it may be understood in terms of
the geometry of the polar decomposition.

Weitsman takes \( X \) to be \( \Hom_\bC (\bC^n, \bC^n) \) with a \(
\Un(n)_\Le \times \Un(n)_\Ri \)-action
\begin{equation*}
  (U,V) \colon A \mapsto UAV^{-1}.
\end{equation*}
The moment map \( \phi \colon X \rightarrow \un(n) \) for the \( \Un(n)_\Ri
\)-action is
\begin{equation*}
  A \mapsto i A^* A
\end{equation*}
which of course is equivariant for the \( \Un(n)_\Ri \)-action and is
invariant for the \( \Un(n)_\Le \)-action.  The image of~\( \phi \) is now
just~\( i\Delta(n) \), where \( \Delta(n) \) is the set of non-negative
Hermitian \( n \times n \) matrices.

We can study \( \phi \) in terms of the polar decomposition of~\( A \).
Write \( A = UP \), where \( P \) is non-negative and \( U \) is unitary.
Now \( P \) is uniquely determined by~\( A \), while \( U \) is uniquely
determined by~\( A \) if \( A \) is invertible.  Explicitly, \( P \) is
the unique non-negative square root~\( H^{1/2} \) of~\( H = A^* A \).  So
knowing \( iH = \phi(A) \) determines~\( P \) but does not impose any
conditions on~\( U \).  Moreover, the map \( H\mapsto H^{1/2} \) is a
section for~\( \phi \).

Thus \( \phi \) maps \( X \) onto~\( i\Delta(n) \) and is a trivial \(
\Un(n) \)-fibration over its interior \( i\Delta(n)^\circ \), which is \( i
\)~times the set of positive Hermitian matrices.  Over the boundary of~\(
i\Delta(n) \), the fibre is \( \Un(n)/\Un(n-k) \), where \( k \) is the
number of positive eigenvalues of~\( H \).  In particular, \( \phi \)~is
injective over the zero matrix.

Let us see what this means in the double fibration
picture~\eqref{eq:double}.  Now the image of~\( \pi_1 \) is \( \mu^{-1} (
-i\Delta(n) + \varepsilon) \), and over \( \mu^{-1}(-i\Delta(n)^\circ +
\varepsilon) \) the map~\( \pi_1 \) is a trivial \( \Un(n) \)-fibration.

So we can view \( \hat M \) as being obtained by removing the complement of
\( \mu^{-1}(-i\Delta(n) +\varepsilon) \) and performing collapsing on the
boundary.  More precisely, if \( H \) has \( k \) positive eigenvalues we
replace \( \mu^{-1}(\varepsilon - iH) \) by \( \mu^{-1}(\varepsilon -
iH)/\Un(n-k) \).  In particular, \( \mu^{-1}(\varepsilon) \) is replaced by
\( \mu^{-1}(\varepsilon)/\Un(n) \).  We thus recover the results of
Weitsman.

\section{A recipe for hyperk\"ahler unitary modifications}
\label{sec:hom}

Let us now consider modifications of hyperk\"ahler manifolds with \( \Un(n)
\) action, for \( n>1 \).  The most obvious choice of~\( X \) in the light of
Weitsman's symplectic construction is the flat hyperk\"ahler vector space
\begin{equation*}
  X = \bH^{n^2} = \Hom_\bC ( \bC^n , \bC^n ) \oplus
  \Hom_\bC (\bC^n, \bC^n).
\end{equation*}
This has a hyperk\"ahler action of~\( \Un(n) \times \Un(n) \) (which we
shall denote, as above, by \( \Un(n)_\Le \times \Un(n)_\Ri \)) given as
follows:
\begin{equation*}
  (U, V) \colon (A, B) \mapsto (UAV^{-1}, VBU^{-1}).
\end{equation*}
Note that the locus where \( \Un(n)_\Ri \) (or~\( \Un(n)_\Le \)) acts
freely includes the open set where \( A \) or~\( B \) is invertible.

\begin{definition}
  \label{def:Hom-mod}
  Let \( M \) be a hyperk\"ahler manifold with tri-Hamiltonian \( \Un(n)
  \)-action and moment map~\( \mu \).  The \emph{hyperk\"ahler
  modification} \( \Mmod \) of~\( M \) at level \( \varepsilon \in \lie
  z\otimes \bR^3 \) with respect to
  \begin{equation*}
    X =\Hom_\bC ( \bC^n , \bC^n ) \oplus \Hom_\bC
    (\bC^n, \bC^n)
  \end{equation*}
  is the hyperk\"ahler quotient
  \begin{equation*}
    (\mu+\phi)^{-1}(\varepsilon)/\Un(n) 
  \end{equation*}
  of \( M \times X \) by~\( \Un(n) \), where \( \Un(n) \) acts on~\( M \)
  by the given action and on~\( X \) by~\( \Un(n)_\Ri \) and \( \phi \) is
  the hyperk\"ahler moment map for the \( \Un(n)_\Ri \)-action on~\( X \):
  \begin{equation}
    \label{eq:phi-hk}
    \phi(A, B) =  \left( \tfrac i2 (A^*A - BB^*), BA \right) \in \un(n)
    \oplus \gl(n,\bC). 
  \end{equation}
\end{definition}

In the formula for \( \phi\colon X\to \un(n)\otimes\bR^3 \), we have chosen
a splitting \( \bR^3 = \bR+\bC \) and used \( \un(n)\otimes \bC =
\gl(n,\bC) \), so \( \phi = (\phi^{\bR},\phi^{\bC}) \).  Note that \( \lie
z = i\bR \Id \), so \( \varepsilon \) is just a point of \( \bR^3 \).

The modification~\( \Mmod \) is hyperk\"ahler with a \( \Un(n) \)-action
induced by the action of~\( \Un(n)_\Le \) on~\( X \).  We will work in the
smooth category, but in general \( \Mmod \) will decompose into a union of
hyperk\"ahler manifolds, cf.~\cite{Dancer-Swann:geometry-qK}.

Starting with a \( \Un(n) \)-orbit \( \Un(n)\cdot m \) in~\( M \), we wish
to understanding the corresponding subset in the modification~\( \Mmod \).
Let \( H \) be the \( \Un(n) \) stabiliser of~\( m \).  We use the notation
of the double fibration picture~\eqref{eq:double}.  The set \(
\hkmod{(\Un(n)\cdot m)} = \pi_2^{}\pi_1^{-1}(\Un(n)\cdot m) \) can be
identified with the quotient of \( \phi^{-1}(\varepsilon-\mu(m)) \) by~\(
H_\Ri \), where \( H_\Ri \) is \( H \)~acting as a subgroup of \(
\Un(n)_\Ri \).  The new \( \Un(n) \)-orbits are the \( \Un(n)_\Le \)-orbits
in this quotient. 

Therefore we need to analyse the fibres \( \phi^{-1}(R,S) \) of the map~\(
\phi \). Note that \( \phi \) is equivariant under the \( \Un(n)_\Ri \)
action and invariant under the \( \Un(n)_\Le \) action, as expected.  

First, we study the linearisation.  This is given by
\begin{equation*}
  d \phi_{(A,B)} \colon (a,b) \mapsto \left( \tfrac i2 (a^* A + A^* a - B b^* -
    b B^*), Ba + b A \right)
\end{equation*}
Let us now take~\( S \) (and hence \( A \), \( B \)) to be invertible.  The
\( \Un(n)_\Ri \) and \( \Un(n)_\Le \) actions are therefore free at~\(
(A,B) \).
 
A vector \( (a,b) \) in the kernel of the complex part \( d \phi^\bC \) of
\( d \phi \) may now be written as \( (hA, -Bh) \) for a unique \( h \in
\gl(n, \bC) \).  The real part of \( d \phi \) acts on this vector by
\begin{equation*}
  d \phi^\bR_{(A,B)} \colon (hA, -Bh) \mapsto A^* (h + h^*) A + B(h + h^*)
  B^*,
\end{equation*}
We obtain

\begin{lemma}
  \label{lem:h}
  Suppose \( A \) and \( B \) are invertible.  The kernel of \( d
  \phi_{(A,B)} \) is the set of vectors \( (hA, -Bh) \) satisfying
  \begin{equation*}
    L (h + h^*) L^* + (h + h^*) =0
  \end{equation*}
  where \( L = B^{-1}A^* \).  In particular, the Hermitian part of \( h \)
  has signature~\( 0 \).  \qed
\end{lemma}

Note that if \( h \) is skew-Hermitian then \( (hA, -Bh) \) is always in \(
\ker d \phi \).  This of course is just the infinitesimal version of the
statement that \( \phi \) is invariant under the action of \( \Un(n)_\Le
\).  Such \( h \) give an \( n^2 \)-dimensional subspace in~\( \ker d \phi
\).

Now \( (A, B) \) is a critical point of~\( \phi \) if and only if \( \dim
\ker d \phi_{(A,B)} > n^2 \).  From Lemma~\ref{lem:h} above, this holds
(with \( A, B \) invertible) if and only if the operator \( \mathbf L = L
\otimes \overline L \) acting on \( \bC^n \otimes (\bC^n)^* = \Hom_\bC
(\bC^n, \bC^n) \) has~\( -1 \) as an eigenvalue with a Hermitian
eigenvector.

Now the eigenvalues of \( \mathbf L \) are just the products of eigenvalues
of~\( L \) with those of~\( \overline L \) (see \cite{Bernstein:matrix}),
so \( -1 \) is an eigenvalue of~\( \mathbf L \) if and only if there is a
complex number~\( \lambda \) such that \( \lambda \) and~\(
-1/\overline{\lambda} \) are both eigenvalues of~\( L \).  The eigenvalues
of~\( \overline L \) then include \( \overline{\lambda} \) and \(
-1/\lambda \), so \( -1 \) actually has multiplicity at least two as an
eigenvalue of \( \mathbf L \) on~\( \Hom_{\bC}(\bC^n,\bC^n) \).  If \( v,w
\) are eigenvectors of \( L \) with eigenvalues \( \lambda,
-1/\overline{\lambda} \) respectively, we see that \( v \otimes w^* + v^*
\otimes w \) is then a Hermitian eigenvector for \( \mathbf L \) with
eigenvalue~\( -1 \).

\begin{theorem}
  A point \( (A, B) \) with \( A,B \) invertible is a critical point for \(
  \phi \) if and only if \( L = B^{-1} A^* \) has a pair of eigenvalues of
  the form \( \lambda, -1/\overline{\lambda} \).

  The set of critical points for \( \phi \) therefore has real codimension
  two in~\( X = \bH^{n^2}\).  The regular points form an open dense set
  in~\( X \). \qed
\end{theorem}

Note that if \( (R,S) \) is a regular value then \( \phi^{-1}(R,S) \) is an
\( n^2 \)-dimensional manifold with an action of \( \Un(n)_\Le \).

\begin{theorem}
  \label{thm:free}
  Let \( (R, S) \) be a regular value of \( \phi \), where \( S \) is
  invertible.  Then \( \Un(n)_\Le \) acts freely and locally transitively
  on the fibre \( \phi^{-1}(R,S) \).  \qed
\end{theorem}

In fact the quotient of the fibre by \( \Un(n)_\Le \) is finite, since \(
\phi \) is a polynomial map in the real and imaginary parts of \( A_{ij},
B_{ij} \).  

\begin{theorem}
  There exist critical values \( (0,S) \) with \( S \)~invertible for which
  the fibre \( \phi^{-1}(0,S) \) has dimension strictly greater than~\( n^2
  \).
\end{theorem}

\begin{proof}
  Start by considering \( n=2 \) and take \( S = L = \diag(1, -1) \); a
  possible choice of point in~\( X \) giving such an~\( L \) is \( (A,B) =
  (L, I) \).  Now \( \mathbf L = L \otimes \overline L \) has a non-trivial
  space of Hermitian eigenvectors for eigenvalue~\( -1 \); in fact such
  eigenvectors are the real anti-diagonal matrices.

  Let us consider the fibre of~\( \phi \) over \( \phi(L, I) = (0, L) \).
  It is convenient to write \( (A,B)\in\phi^{-1}(0,L) \) as \( A = UT \)
  and~\( B = \tilde TV \), where \( U \), \( V \) are unitary and \( T \),
  \( \tilde T \) are lower and upper triangular respectively with real
  positive entries on the diagonal.

  The real equation \( \phi^\bR(A,B)=0 \) is \( A^* A = BB^* \) which
  becomes \( T^*T = \tilde T{\tilde T}^* \), and it is easy to check this
  implies \( \tilde T=T^* \).

  The complex equation \( \phi^\bC(A,B)=L \), is now the condition \(
  T^*VUT = L \), i.e., the requirement that \( (T^*)^{-1}LT^{-1} \) be
  unitary.  Each such \( T \) gives a unique \( (A,B) \) up to the action
  of~\( \Un(2)_\Le \).  Writing
  \begin{equation*}
    T =
    \begin{pmatrix}
      a & b \\
      0 & d
    \end{pmatrix}
    ,
  \end{equation*}
  where \( a, d \in (0, \infty) \) and \( b \in \bC \), this is equivalent
  to the equations
  \begin{gather*}
    d^2 (d^2 + \abs b^2) = a^4 d^4, \\
    bd (a^2 - \abs b^2 - d^2) = 0, \\
    d^2 \abs b^2 + (a^2 - \abs b^2)^2 = a^4 d^4.
  \end{gather*}
  If \( b \neq 0 \) the second equation gives \( d^2 = a^2 - \abs b^2 \)
  and the remaining equations are now both equivalent to \( a^2 d^2 = 1 \).
  So we have
  \begin{equation*}
    a^2 = \frac1{d^2},\quad \abs b^2 = \frac1{d^2} - d^2,\quad d \in (0, 1).
  \end{equation*}
  So we have a set of solutions parametrised by the punctured disc.  If \(
  b=0 \) then \( a=d=1 \), corresponding to the centre of the disc.

  Thus the quotient by the free \( \Un(2)_\Le \)-action on the fibre \(
  \phi^{-1}(0, L) \) is diffeomorphic to the open disc in \( \bC \).

  For general \( n \), consider \( S = L \oplus L' \) with \( L' \) an
  invertible \( (n-2)\times (n-2) \) matrix.  The analysis above shows that
  the fibre is determined by the upper triangular matrices \( Y = T^{-1} \)
  with positive entries on the diagonal such that \( Y^*SY \) is unitary
  and each such matrix determines a unique \( \Un(n)_\Le \)-orbit in the
  fibre.  The result follows by taking direct sums of solutions for the \(
  n=2 \) case with solutions for~\( L' \).
\end{proof}

We may also get fibres with positive-dimensional quotient by \( \Un(n)_\Le
\) if \( S \) is non-invertible.

\begin{theorem}
  \label{thm:zero}
  The quotient \( \phi^{-1}(0,0)/\Un(n)_\Le \) may be identified with the
  set of non-negative Hermitian matrices of rank at most~\( n/2 \).
\end{theorem}

\begin{proof}
  Write \( (A,B) \in \phi^{-1}(0,0)\) as \( A = UP, B = QV \) where \( U, V
  \) are unitary and \( P, Q \) are non-negative Hermitian matrices (that
  is, we are taking left and right polar decompositions of \( A \) and \( B
  \)).  The real equation \( A^* A = BB^* \) now says \( P=Q \), and the
  complex equation becomes
  \begin{equation*}
    PVUP=0.
  \end{equation*}
  This holds precisely when \( \im P \) may be mapped by a unitary \( W=VU
  \) transformation into \( \ker P \), which is true if and only if \(
  \rank(P) \leqslant n/2 \).
  
  Given \( P \) non-negative Hermitian of rank at most \( n/2 \), it
  remains to show that there is a unique \( U(n)_\Le \)-orbit in \(
  \phi^{-1}(0,0) \) determined by~\( P \).  For \( (A_1,B_1), (A_2,B_2) \in
  \phi^{-1}(0,0) \) with \( A_1=U_1P, B_1=PV_1 \) and \( A_2=U_2P, B_2=PV_2
  \) for unitary matrices \( U_1,V_1,U_2,V_2 \) we have that \( W_1=V_1U_1
  \) and \( W_2=V_2U_2 \) both represent transformations that map \( \im P
  \) into \( \ker P = (\im P)^\bot\).

  Let \( X = \im P \), \( Y = W_1^{-1}\im P \) and \( Z = W_2^{-1}\im P \).
  Then \( Y,Z \leqslant X^\bot \) and \( W_2^{-1}W_1^{} \) maps \( Y \)
  isometrically on to~\( Z \), so there is a unitary transformation \(
  w\in\Un(n) \) with \( w|_X = \Id_X \) and \( w|_Y = W_2^{-1}W_1^{}|_Y \).
  Put \( u= U_1w^{-1}U_2^{-1} \) which is unitary.  Then \( B_1 = PV_1 =
  PW_1 U_1^{-1} = PW_2 wU_1^{-1} = B_2 u^{-1}\) and \( u A_2 = U_1 w^{-1} P
  = U_1 P = A_1 \), giving \( (A_1,B_1) = (u,1)\cdot (A_2,B_2) \), as
  required.
\end{proof}

The existence of fibres of~\( \phi \) of dimension greater than~\(n^2 \),
and more generally fibres whose quotient by~\( \Un(n)_\Le \) is of positive
dimension, is a phenomenon that does not occur in the cutting/modification
constructions for symplectic or hyperk\"ahler manifolds with torus actions,
or even in the case of symplectic manifolds with \( \Un(n) \) actions.
Indeed, as noted earlier, in these cases the action of the group on the
fibres is transitive.

This behaviour means that when we perform the \( \Un(n) \) hyperk\"ahler
modification, certain loci in \( M \) are being blown up.  In particular
the set \( \mu^{-1}(\varepsilon) \), where \( \varepsilon \) is the level
at which we perform the modification, is not simply collapsed by \( \Un(n)
\) in the modification, as \( \phi \) is not injective over the origin. As
noted later in the section, the corresponding region in \( \Mmod \)
includes \( \mu^{-1}(\varepsilon)/\Un(n) \) but also includes other strata
corresponding to nonzero elements of \( \phi^{-1}(0,0) \).  This contrasts
with the Abelian case or the symplectic \( \Un(n) \) case, where \( \phi \)
is injective over the origin.

\begin{remark}
  Using the techniques of the proof of Theorem~\ref{thm:zero}, one may
  check that \( \phi^{-1}(0,\lambda \Id) \) for each \( \lambda\ne0 \)
  consists exactly of one \( \Un(n)_\Le \)-orbit.
\end{remark}

\begin{remark}
  We saw above that the moment map \( \phi \) for the \( \Un(n)_\Ri \)
  action on \( X \) can have critical points, and fibres of larger than
  expected dimension, even on the locus where the action is free.  Again,
  this is a new feature that appears for non-Abelian hyperk\"ahler actions.

  In the symplectic case, the kernel of the differential of a moment map is
  just the symplectic orthogonal of the space \( \mathcal G \) of Killing
  fields for the action.  In the K\"ahler setting, this is of course the
  metric orthogonal \( (I\mathcal G)^\perp \) where \( I \) denotes the
  complex structure.  Hence on the locus where the action is free the rank
  is maximal because \( \mathcal G \) has the maximal dimension \( \dim G
  \).

  For hyperk\"ahler moment maps, the kernel is now the intersection of the
  symplectic orthogonals to \( \mathcal G \) relative to the K\"ahler forms
  \( \omega_I, \omega_J, \omega_K \); equivalently it is the metric
  orthogonal
  \begin{equation*}
    (I\mathcal G + J\mathcal G + K\mathcal G)^\perp
  \end{equation*}
  If the action is \emph{Abelian} then the moment map is invariant under
  the group action and hence \( \mathcal G \) is contained in the above
  orthogonal complement.  It follows that \( I\mathcal G, J\mathcal G \)
  and \( K\mathcal G \) are mutually orthogonal and the above sum is
  direct.  Again, we see that on the free locus the moment map has maximal
  rank.  In the non-Abelian case this can fail, leading to critical points
  on the free locus.  Note that the value of the moment map at such
  critical points is non-central.
\end{remark}

In contrast to the Abelian case, the moment map \( \phi \) is not
surjective.  For later use, we will prove a little more.

\begin{theorem}
  \label{thm:image}
  For each \( \varepsilon\in\lie z\otimes \bR^3 \) there exist values \(
  \varepsilon + (R,S) \) with \( (R,S) \in \su(n) + \Sl(n,\bC) \) that are
  not in the image of \( \phi \).
\end{theorem}

\begin{proof}
  Rotating the choice of complex structures on \( X \), we may assume that
  \( \varepsilon = (0,\lambda \Id) \) for some complex number~\( \lambda
  \).

  Consider \( (0, M) \in \un(n) \oplus \gl(n, \bC) \), where \( M =
  (m_{ij}) \) is any matrix that has:
  \begin{compactenum}
  \item\label{item:first} first column zero: \( m_{i1}= 0 \) for all \(
    i=1,\dots,n \),
  \item\label{item:nz} non-zero entries immediately above the diagonal: \(
    m_{i,i+1} \ne 0 \), \( i=1,\dots,n-1 \),
  \item\label{item:ze} all other entries above the diagonal zero: \(
    m_{i,j} = 0 \) for \( j>i+1 \), and
  \item\label{item:tr} the trace \( m_{22}+\dots+m_{nn}\) equals
    \( n\lambda \).
  \end{compactenum}
  Note that \( M \) has rank \( n-1 \).

  If \( \phi(A,B) = (0,M) \), the equations are
  \begin{equation*}
    A^*A - BB^* = 0,\quad BA = M.
  \end{equation*}
  The first of these implies, using \( \ker A = \ker A^*A \) etc., that \(
  A \) and~\( B \) have equal rank.  By the second equation, this rank must
  be \( n-1 \).  Write
  \begin{equation*}
    B =
    \begin{pmatrix}
      \mathbf v_1^T \\
      \vdots\\
      \mathbf v_n^T
    \end{pmatrix}
    ,\quad
    A =
    \begin{pmatrix}
      \mathbf w_1 & \dots & \mathbf w_n
    \end{pmatrix}
  \end{equation*}
  for column vectors \( \mathbf v_i \) and \( \mathbf w_j \).  Now we have
  \begin{equation*}
    BA =
    \begin{pmatrix}
      \mathbf v_i^T\mathbf w_j
    \end{pmatrix}
    = M.
  \end{equation*}
  Condition \itref{item:nz} gives that \( \mathbf v_1,\dots,\mathbf v_{n-1}
  \) and \( \mathbf w_2,\dots,\mathbf w_n \) are non-zero.  We claim that
  each of these collections of \( n-1 \) vectors is linearly independent.  To
  see this, suppose \( 0 = a_2\mathbf w_2+\dots+a_n\mathbf w_n \).  Then
  multiplying by \( \mathbf v_1^T \) and using \itref{item:ze}, we get \( 0
  = a_2\mathbf v_1^T\mathbf w_2 \) and so \( a_2=0 \).  Multiplying
  successively with \( \mathbf v_i^T \), \( i=2,\dots,n-1 \) then shows \(
  a_j=0 \) for each \( j \), as required.  A similar argument gives the
  linear independence of \( \mathbf v_1,\dots,\mathbf v_{n-1} \).

  Since \( \rank A = n-1 \), we have that \( \mathbf w_1 \) is a linear
  combination of the vectors \( \mathbf w_2,\dots,\mathbf w_n \).
  Multiplying this combination successively by \( \mathbf v_i^T \), \(
  i=1,\dots,n-1 \), and using \itref{item:first}, we find this combination
  is zero.  Thus \( \mathbf w_1=0 \) and \( A^*A \) has leading entry~\( 0
  \).  The corresponding entry in \( BB^* \) is \( \norm{\mathbf v_1}^2 \),
  which is non-zero.  Thus there is no \( (A,B) \) mapping to \( (0,M) \)
  under \( \phi \).  Taking \( R=0 \) and \( S = M - \lambda\Id \), we have
  by \itref{item:tr} \( S \in \Sl(n,\bC) \) and the result follows.
\end{proof}

Note that diagonal matrices of the form \( D = \diag(0,d_2,\dots,d_n) \)
may be obtained as the limit of the matrices \( M \) in the above proof.
However the point \( (0,D) \) is in the image of~\( \phi \), putting \(
d_1=0 \) and writing \( d_j = e^{i\theta_j}r_j^2 \) with \( r_j\geqslant 0
\), we have \( (0,D) = \phi(UP,P) \) where \( U = \diag(e^{i\theta_j}) \)
and \( P=\diag(r_j) \).  Thus we have:

\begin{corollary}
  The image of \( \phi \) is not open.  \qed
\end{corollary}

The above results are easiest to summarise when the action of \( \Un(n) \)
on \( M \) is free.  An example of such a hyperk\"ahler manifold is the
space \( T^*\GL(n,\bC) \) discussed in the final section of the paper.

\begin{theorem}
  Suppose \( M \) is a hyperk\"ahler manifold with free tri-Hamiltonian \(
  \Un(n) \)-action, \( n>1 \), with moment map \( \mu \).  Let \( \Mmod \)
  be the hyperk\"ahler modification at level \( \varepsilon \) defined in
  Definition~\ref{def:Hom-mod}.  Then \( \Mmod \) is a smooth hyperk\"ahler
  manifold with tri-Hamiltonian \( \Un(n) \)-action.  Moreover, in \( \Mmod
  \)
  \begin{enumerate}
  \item\label{item:remove} the set \( \mu^{-1}([\un(n)\otimes\bR^3]
    \setminus [\varepsilon -\im\phi]) \) is removed from \( M \),
  \item\label{item:fibration} there is an open set \( \mathcal U\subset
    \un(n)\otimes \bR^3 \) and a manifold of dimension \( \dim M + n^2 \)
    that fibres over both \( \mu^{-1}(\mathcal U) \) and \(
    \hkmod{(\mu^{-1}(\mathcal U))} \); up to finite covers both of these
    fibrations are principal \( \Un(n) \)-bundles,
  \item\label{item:quotient} \( \Mmod \) contains a copy of the
    hyperk\"ahler quotient of \( M \) at level \( \varepsilon \),
  \item\label{item:blowup} there is a set \( \mathcal Z\subset
    \un(n)\otimes \bR^3 \) such that \( \hkmod{(\mu^{-1}(\mathcal Z))} \)
    has larger dimension than \( \mu^{-1}(\mathcal Z) \) and is `blown-up'.
  \end{enumerate}
\end{theorem}

\begin{proof}
  The diagonal action of \( \Un(n) \) on \( M\times X \) is free, so
  reduction at a central value \( \varepsilon \) gives a smooth
  hyperk\"ahler quotient.  The moment map~\( \phi_\Le \) for \( \Un(n)_\Le
  \) on \( X \) descends to the quotient giving a moment map for the
  induced \( \Un(n) \)-action; explicitly
  \begin{equation*}
    \phi_\Le(A,B) = (\tfrac12(B^*B-AA^*,AB)).
  \end{equation*}

  Part~\itref{item:remove} is immediate: by
  Theorem~\ref{thm:image} the set being removed can be nonempty.

  For \itref{item:fibration}, the open set \( \mathcal U \) is given by
  Theorem~\ref{thm:free} and the comment immediately following it.

  The hyperk\"ahler quotient of \( M \) at level \( \varepsilon \) is \(
  \mu^{-1}(\varepsilon)/\Un(n) \) may be identified with the quotient of
  points \( (m,p)\in M\times X \), with \( \mu(m)=\varepsilon \) and \( p=0
  \) by the diagonal \( G \)-action, so is a subset of \( \Mmod \).  This
  gives part~\itref{item:quotient}.

  The set \( \mathcal Z \) in \itref{item:blowup} consists of those points
  where the fibre of \( \phi \) has dimension greater than \( n^2 \).
\end{proof}

We note that the intertwining of the actions on \( M \) and \( X \) imply
that the \( \Un(n) \)-action on \( \Mmod \) can be effective.  This is
despite the fact that the diagonal \( \Un(1) \)-subgroup of \( \Un(n)_\Le
\times \Un(n)_\Ri \) acts trivially on~\( X \).

\begin{example}
  \label{ex:ADHM}
  Let us consider an example of our construction when \( M \) is flat.  We
  take \( M = \Hom_\bC (\mathbb C^n, \bC^n) \oplus \Hom_\bC (\bC^n, \bC^n)
  \), but this time with \( \Un(n) \) acting by conjugation.  The moment
  map \( \mu = (\mu^\bR, \mu^\bC) \) is given by
  \begin{equation*}
    \mu \colon (B_1, B_2) \mapsto ( [ B_1, B_1^*] + [B_2, B_2^*], [B_1, B_2])
  \end{equation*}
  The modification \( \Mmod \) at level \( \varepsilon \) will now be the
  quotient of the space
  \begin{equation*}
    \{\, (B_1, B_2, A, B) : [ B_1, B_1^*] + [B_2, B_2^*] + A^* A - BB^* =
    \varepsilon^\bR,
    [B_1, B_2] + BA = \varepsilon^\bC \,\}
  \end{equation*}
  by the action
  \begin{equation*}
    (B_1, B_2, A, B) \sim (gB_1g^{-1}, gB_2g^{-1}, Ag^{-1}, gB),\qquad g
    \in \Un(n)
  \end{equation*}
  This is just the deformed instanton space \( {\mathcal M}_{\varepsilon}
  (n,n) \) of Nakajima (see for example \cite{Nakajima:Hilbert}).
\end{example}

\begin{remark}
  Consider taking the hyperk\"ahler quotient of \( X \) by the central~\(
  \Un(1) \), which is the same for both \( \Un(n)_\Le \) and \( \Un(n)_\Ri
  \).  If we choose a non-zero central level then we obtain the
  hyperk\"ahler structure found by Calabi on \( Y = T^*\CP(m) \) in the
  special case \( m=n^2-1 \).  This space inherits tri-holomorphic actions
  of \( \SU(n)_\Le \) and \( \SU(n)_\Ri \), so may be used to construct
  modifications of hyperk\"ahler manifolds with \( \SU(n) \) symmetry.  We
  see all the features described above.  In particular,
  Theorem~\ref{thm:image} shows that parts of \( M \) are cut away in this
  process.  The \( \SU(n) \)-modification is less flexible than the \(
  \Un(n) \) case, since one can only reduce at the level \( \varepsilon = 0
  \).
\end{remark}

\begin{remark}
  As in the symplectic case \cite{Weitsman:cuts}, we may generalise the
  modification construction by taking \( X \) to be hyperk\"ahler with a \(
  H \times G \) action.  Now, performing our construction will give a
  hyperk\"ahler space with an action of \( H \) rather than \( G \).  For
  example, we may take \( H = \Un(r) \) and \( X = \Hom_\bC (\bC^n, \mathbb
  C^r) \oplus \Hom_\bC (\bC^r, \bC^n) \).  The moment map \( \phi \) is as
  in equation~\eqref{eq:phi-hk}, but \( A, B \) are now \( r \times n \)
  and \( n \times r \) matrices respectively.

  Applying this to \( M \) as in example \ref{ex:ADHM} gives the general
  Nakajima space \( {\mathcal M}_{\varepsilon}(r,n) \).
\end{remark}

\section{A gauge theory modification}
\label{sec:gauge}

In this section we describe another possible choice of~\( X \).  This is
less involved topologically, but it does gives a metric deformation and it
applies to all compact symmetry groups~\( G \).

Recall from \cite{Kronheimer:cotangent} that if \( G \) is compact then the
cotangent bundle \( T^* G_\bC \) of the complexification of \( G \) carries
a complete hyperk\"ahler structure preserved by an action of \( G \times G
\).  Kronheimer proved this by identifying the cotangent bundle with the
moduli space of solutions to Nahm's equations
\begin{equation*}
  \frac{d T_i}{dt} + [T_0, T_i] = [T_j, T_k],
\end{equation*}
\( (i,j,k) \) a cyclic permutation of \( (123) \), where the \( T_a \) are
smooth maps from \( [0,1] \) to \( \g \).  The moduli space is formed by
identifying any two such solutions which are equivalent under the action of
the restricted gauge group
\begin{equation*}
  \mathsf G_0^0 = \{\, g \colon [0,1] \mapsto G : g(0)=g(1)= \Id \,\}
\end{equation*}
by
\begin{equation*}
  T_0 \mapsto gT_0 g^{-1} - \frac{dg}{dt} g^{-1},\qquad T_i \mapsto g T_i
  g^{-1} \quad (i=1,2,3).
\end{equation*}
Denoting by \( \mathsf G \) the group of gauge transformations \( g \colon
[0,1] \rightarrow G \) with no endpoint restrictions, we have an action of
\( \mathsf G/\mathsf G_0^0 \cong G \times G \) on the moduli space
preserving the hyperk\"ahler structure.  Denote by \( G_\Le, G_\Ri \)
respectively the copies of \( G \) corresponding to gauge transformations
equal to the identity at \( t=1 \) (resp., \( t=0 \)).  We showed in
\cite{Dancer-Swann:compact-Lie} that \( G_\Le \) and \( G_\Ri \) act
freely.

It is useful to observe that a point in the moduli space may be gauged by a
unique transformation \( g \in G_\Le \) to a configuration with \( T_0 =0
\).  (We just solve the linear equation \( dg/dt = g T_0 \)).  Now \( (T_1,
T_2, T_3) \) satisfy the reduced Nahm's equations
\begin{equation*}
  \frac{d T_i}{dt} = [T_j, T_k],\qquad\text{\( (i,j,k) \) cyclic  permutation
  of \( (123) \)}
\end{equation*}
and solutions to these equations are determined by \(
(T_1(1),T_2(1),T_3(1)) \).

Given a hyperk\"ahler \( G \)-manifold \( M \), we can form a modification
\( \Mmod \) by taking \( X \) to be \( T^* G_\bC \) with the above \( G
\times G \) action.

In \cite{Dancer-Swann:compact-Lie} we proved that the moment map \( \psi \)
for the \( G_\Ri \) action is given by evaluation of \( (T_1, T_2, T_3) \)
at \( t=1 \).  Note that \( \psi \) is \( G_\Ri \)-equivariant and \( G_\Le
\)-invariant.  The image of~\( \psi \) is an open set \( U_{\g} \) in \(
\g^* \otimes \bR^3 \), consisting of the triples \( (X_1, X_2, X_3) \) such
that there exists a solution to the Nahm equations smooth on \( [0,1] \)
with \( T_i(1) = X_i \), \( (i=1,2,3) \).  The results of
\cite{Dancer-Swann:compact-Lie} show that \( U_{\g} \)~is star-shaped with
respect to the origin.  This is because the Nahm equations are preserved
under the affine reparametrisation \( T_t \mapsto a T_i(at + 1-a) \), \( 0
< a < 1 \), which scales \( T_i(1) \) by~\( a \) while preserving the
condition of smoothness on~\( [0,1] \).  Moreover \( U_{\g} \) is a proper
subset of \( \g^* \otimes \bR^3 \) if \( G \)~is non-Abelian.  This is
because we may find solutions to the Nahm equations that are smooth at \(
t=1 \) but blow up at some point in the interior of \( [0,1] \).

From above, given \( \psi( T_0, T_1, T_2, T_3) \) we know the solutions to
the reduced Nahm equations are gauge equivalent to \( (T_0, T_1, T_2, T_3)
\), hence we know \( (T_i) \) up to the (free) \( G_\Le \) action.  So \(
\psi \colon X \rightarrow U_{\g} \subset \g^* \otimes \bR^3 \) is a \( G
\)-fibration, and this fibration is trivial (e.g.\ because \( U_{\g} \) is
star-shaped).

We see that the modification of \( M \) using \( X = T^* G_\bC \) may be
viewed topologically as removing the complement of \( \mu^{-1}(-U_{\g} +
\varepsilon) \) from~\( M \).  We do not perform any collapsing because
there are no special fibres of~\( \psi \).  Note that the metric on \(
\Mmod \) will be complete as long as \( M \) is complete.

Note that if \( G = S^1 \) then \( X \) is just \( T^* \bC^* \), which can
be identified with \( \bR^3 \times S^1 \) with the flat hyperk\"ahler
metric.  As \( G \) is now Abelian, constant gauge transformations act
trivially, so the \( G \times G \)-action on~\( X \) just collapses to a \(
G \)-action which is rotation of the \( S^1 \) factor.  We may form the
modification \( \Mmod \) in the same way as in \cite{Dancer-S:mod}.  Now \(
\psi \) is just projection of this trivial \( S^1 \)-bundle onto \( \bR^3
\), so \( \Mmod \) is diffeomorphic to \( M \).  However the metric on \(
\Mmod \) may be different from that on \( M \).  For example, if \( M \) is
flat \( \bR^4 \) then \( \Mmod \) is \( \bR^4 \) with the Taub-NUT metric.
So we may view \( \Mmod \) as a metric deformation of \( M \) (note that
Bielawski \cite{Bielawski:tri-Hamiltonian} has an alternative way of
expressing Taub-NUT deformations if \( M \) is \( 4n \)-dimensional with
the action of an \( n \)-dimensional Abelian group).  As any compact \( G
\) contains such circle subgroups we find

\begin{theorem}
  Let \( M \) be a hyperk\"ahler manifold with tri-Hamiltonian \( G
  \)-action and moment map~\( \mu \), for some compact Lie group~\( G \).
  Then the modification of \( M \) via \( X = T^*G_{\bC} \) at a central
  level~\( \varepsilon \) is diffeomorphic to~\(
  \mu^{-1}(-U_{\g}+\varepsilon) \) but need not be isometric to this
  subset of \( M \).  If \( M \) is complete, then this modification is
  also complete.\qed
\end{theorem}

\begin{remark}
  Some analogous results may be obtained in hypersymplectic geometry.
  Hypersymplectic cuts for circle actions were introduced in
  \cite{Dancer-S:hs-survey} and some results in the hypersymplectic
  non-Abelian case have been found by T.~Matsoukas in his 2009 Oxford
  D.Phil.\ thesis \cite{Matsoukas:hypersymplectic}.
\end{remark}

\providecommand{\bysame}{\leavevmode\hbox to3em{\hrulefill}\thinspace}
\providecommand{\MR}{\relax\ifhmode\unskip\space\fi MR }
\providecommand{\MRhref}[2]{%
  \href{http://www.ams.org/mathscinet-getitem?mr=#1}{#2}
}
\providecommand{\href}[2]{#2}

\begin{small}
  \setlength{\parindent}{0pt} A.~S. Dancer:

  Jesus College, Oxford, OX1 3DW United Kingdom.

  \textit{E-mail}: \url{dancer@maths.ox.ac.uk}

  \medskip A.~F. Swann:

  Department of Mathematics and Computer Science, University of Southern
  Denmark, Campusvej 55, DK-5230 Odense M, Denmark

  \textit{and}

  CP\textsuperscript3-Origins, Centre of Excellence for Particle Physics
  Phenomenology, University of Southern Denmark, Campusvej 55, DK-5230
  Odense M, Denmark

  \textit{E-mail:} \url{swann@imada.sdu.dk}
\end{small}

\end{document}